\documentclass[11pt]{amsart}
\usepackage{lmodern}
\usepackage[a4paper,margin=3cm]{geometry}
\usepackage{amsmath,amssymb}
\usepackage[utf8]{inputenc}
\usepackage[T1]{fontenc}
\usepackage[english]{babel}
\usepackage{graphicx}
\usepackage{xcolor}
\usepackage[all,cmtip]{xy}
\usepackage{cancel}
\usepackage{ulem}
\SelectTips{lu}{11}
\usepackage{setspace}
\usepackage{subfigure,graphics,graphicx,psfrag}
\usepackage{tikz} 
\usepackage{overpic}

\theoremstyle{plain}
\newtheorem{theorem}{Theorem}[section]
\newtheorem{proposition}[theorem]{Proposition}
\newtheorem{corollary}[theorem]{Corollary}
\newtheorem{lemma}[theorem]{Lemma}

\theoremstyle{remark}
\newtheorem{remark}{Remark}

\theoremstyle{definition}
\newtheorem{definition}[theorem]{Definition}

\title{ \bf Separation properties of codimension-1 maps between generalized manifolds}

\author[E.\ L.\ dos Santos]{Edivaldo L.\ dos Santos}
\address{Departamento de Matem\'{a}tica \\
	Federal University of S\~ao Carlos \\
	Rodovia Washington Luiz. km 235 \\
	S\~ao Carlos, SP, Brazil. \textit{E-mail address:} \rm edivaldo@ufscar.br  \textit{ORCID:} https://orcid.org/0000-0001-9046-1473}

    \author[ Telmo A. ]{Telmo I. Acosta Vellozo }
\address{ CENUR Noreste, Universidad de la Rep\'ublica - Uruguay \textit{E-mail address:} \rm telmo.acosta@cur.edu.uy \textit{ORCID:} https://orcid.org/0009-0004-2692-5375
}


\thanks{This work is partially supported by the Projeto Tem\'atico: Topologia Alg\'ebrica,Geom\'etrica e Diferencial, FAPESP Process Number 2022/16455-6. This study was financed in part by the Coordenação de Aperfeiçoamento de Pessoal de Nível Superior - Brasil (CAPES) - Finance Code 001.}

\begin{document}
\baselineskip=1.7em

\maketitle

\noindent{\bf Abstract.}
In this work, we obtained separation results via codimension-1 maps to generalized manifolds. More specifically, we proved results that allow us to estimate the number of connected components of the complement of the image of such maps.

\section{Introduction}
\par This work focuses on obtaining results for generalized manifolds through generalizations of the results obtained for smooth manifolds and topological manifolds. Given a continuous map $f:M\rightarrow N$ between generalized manifolds, it is often an important problem to study the topology of the complement $N-f(M)$ of $f(M)$ in $N$. Here we consider the codimension 1 case; i.e., the case where  $\dim N- \dim M=1$, and study the number of connected components of $N-f(M)$. Such a problem was first considered by Vaccaro in \cite{Vaccaro}, who found a PL-immersed $S^2$ in $\mathbb{R}^3$ whose complement is connected. On the other hand, it is known by the Jordan-Brouwer Theorem that the number of connected components of the complement of an embedded $(n-1)$-sphere in $\mathbb{R}^n$ is equal to 2. Another important result in this way was obtained by Feighn in \cite{Feighn}, taking that if $H_1(N;\mathbb{Z}_2)=0$, then every proper $C^2$-immersion $f:M\rightarrow N$ disconnects $N$ where $M$ and $N$ are smooth manifolds. 
\par For topological manifolds, there are many results with different approaches. In \cite{Ballesteros1}, Nu\~no Ballesteros and Romero Fuster showed that, if $H_1(N;\mathbb{Z}_2)=0$, then every $f:M\rightarrow N$ proper continuous map whose $A(f)=\{x\in M: f^{-1}f(x)\neq x\}$ the self-intersection set is not dense in any connected component  of $M$, the complement $N-f(M)$ is disconnected. In \cite{Ballesteros}, Nu\~no Ballesteros with the same hypothesis as in \cite{Ballesteros1} and $N-f(A)$ connected, gave a formula for the number of connected components of $N-f(M)$ in terms of the \v Cech cohomology.

\par Another result in this direction for smooth manifolds using normal crossing points and the primary obstruction to topological embeddings was obtained by Biasi and Saeki in \cite{Biasi3}.\\
\par In this work, we present two main separation results involving codimension-1 maps in generalized manifolds. 

In Section 2, we recall the fundamental definitions and properties of generalized manifolds.  The section emphasizes that generalized manifolds need not be homeomorphic to topological manifolds, and it reviews key duality theorems that extend to this setting, namely Poincaré duality and Alexander duality.

In Section 3, we establish the first main separation result for codimension-1 maps into generalized manifolds, with the results depending on the structure of the self-intersection set and extending those of Ballesteros \cite{Ballesteros}. Theorem 3.2 provides a formula for the number of connected components of the complement in terms of Čech cohomology, while Corollary 3.3 yields a Jordan–Brouwer  theorem for generalized manifolds.

Finally, in Section 4, we develop the second main separation result for codimension-1 maps into generalized manifolds, based on the primary obstruction to topological embeddings. After introducing the obstruction class $\theta(f)$ and its properties, we prove Theorem 5.7, which shows that under suitable hypotheses (including $H_1(N)=0$ and $w_1(f)=0$), the complement $N - f(M)$ has at least three connected components. This extends separation results of Biasi and Saeki in \cite{Biasi3}.

\section{Generalized Manifold}
\begin{definition}
    A locally compact space $X$ is a generalized $m$-manifold if the following conditions are satisfied:
    \begin{enumerate}
        \item $X$ is an ENR (Euclidean Neighborhood Retract), if there exist a subspace $Y$ of some $\mathbb{R}^{n}$ homeomorphic to $X$, a neighborhood $V$ of $Y$ and a retraction $r:V\longrightarrow Y$;
        \item $H_{\ast}(X,X-\{x\};R)\simeq H_{\ast}(\mathbb{R}^{m},\mathbb{R}^{m}-\{0\};R)$ for every $x\in X$, where $R$ is $\mathbb{Z}$ or $\mathbb{Z}_2$.
    \end{enumerate}
    The space $X$ is a generalized $m$-manifold with boundary if the condition $2$ is replaced by $H_m(X, X-\{ x\};R)\simeq R$ or $0$, and if $bd(X)=\{x\in X: H_m(X, X-\{ x\};R)\simeq 0\}$ is a generalized $(m-1)$-manifold embedded in $X$
\end{definition}

Such manifolds have been studied throughout the second half of the twentieth century, for example, it is well known that there are generalized manifolds that are not homeomorphic to topological manifolds(see \cite{Bryant0}). Another known fact is that the following duality results hold (see \cite{Bredon} \cite{Denise}).

\begin{theorem}[Poincar\'e's Duality Theorem for Generalized Manifold]
  The duality map\\ $D_M : H_c^k(M;R)\rightarrow H_{n-k}(M;R)$ given by $D_M(\alpha)= \alpha \frown [M]$, is an isomorphism for all $k$ whenever $M$ is an $R$-oriented generalized $n$-manifold.  
\end{theorem}

\begin{theorem}[Alexander's Duality Theorem for Generalized Manifold]
  Let $X$ be an oriented generalized $n$-manifold. If $B$ is a closed subset of $X$, then $\check{H}_c^{n-i}(B;R)\simeq H_i (X,X-B;R)$ for each integer $i$, where $\check{H}_c^{\ast}$ denotes \v Cech cohomology with compact support.
\end{theorem}

We now present the main results concerning the topology of generalized manifolds. These results are obtained under two assumptions, which we define below.

\begin{definition}
    A generalized $n$-manifold, $n\geq 5$, has the disjoint disks property (DDP) if given $\epsilon>0$ and a pair of maps $f,g:D^{2}\rightarrow X$, there is a pair of maps $f',g':D^{2}\rightarrow X$ such that $d(f,f')<\epsilon$, $d(g,g')<\epsilon$ and $f'(D^2)\cap g'(D^2)= \emptyset$.
\end{definition}
\begin{definition}
    A resolution for a generalized $n$-manifold $X$ is a map $\phi :M\rightarrow X$ such that $\phi |_{\phi^{-1}(U)}: \phi^{-1}(U)\rightarrow U$ is a homotopy equivalence for all open $U\supset X$, where $M$ is a topological manifold. If $X$ admits a resolution, we say that $X$ is resolvable.
\end{definition}

Now, we are in condition to state one of the most important theorems for generalized manifolds.

\begin{theorem}[Edwards's approximation theorem]
 Let $X$ be a generalized $n$-manifold with DDP. If $\phi :M\rightarrow X$ is a resolution for $X$, then $\phi$ is the limit of a sequence of homeomorphisms $h_i:M \rightarrow X$.   
\end{theorem}

\begin{corollary}\label{Edwards cor}
Let $X$ be a resolvable generalized $n$-manifold with $n \geq 5$. Then $X$ is a topological manifold if and only if $X$ has the DDP.
\end{corollary}

In this way, Corollary \ref{Edwards cor} provides a criterion for characterizing topological manifolds among generalized manifolds. In \cite{Quinn, Quinn1}, F. Quinn associates to any connected generalized $n$-manifold, with $n \geq 4$, a local index $\iota(X) \in 1 + 8\mathbb{Z}$ and shows that $\iota(X) = 1$ if and only if $X$ is resolvable. Combining this with Edwards's Approximation Theorem, we obtain the following characterization theorem for topological manifolds.

\begin{theorem}[Edwards-Quinn]
  If $n\geq 5$, a generalized $n$-manifold $X$ with $DDP$ is a topological manifold if and only if $\iota(X)=1$.  
\end{theorem}

This characterization raises the question: Are there generalized manifolds $X$ with $\iota(X)\neq 1$? The following theorem, which appeared \cite{Bryant0}, answers this question.

\begin{theorem}[Bryant, Ferry, Mio, Weinberger]
  Let $M$ be a closed simply-connected topological $n$-manifold and $\sigma\in 8\mathbb{Z}+1$. There is a generalized $n$-manifold $X$ homotopy equivalent to $M$ with $\iota(X)=\sigma$.  
\end{theorem}

\section{Separation by codimension-1 map depending on the self-intersection set }

In this section, we will see the first result of separation by codimension-1 map to generalized manifolds, Theorem \ref{teoballesteros}, which is an extension of Theorem 2.2 in \cite{Ballesteros}, from topological manifold to generalized manifold.\newline
\par The following result was proven in \cite{Ballesteros}.
\begin{lemma}\label{lemaball}
Consider the following commutative diagram of R-modules, where the rows are exact and $g$ is an isomorphism
$$\xymatrix{& A\ar[r]\ar[d]^{f}& B\ar[r]\ar[d]^{g}& C\ar[r]\ar[d]^{h}& 0 \\ D\ar[r]^{\lambda}& A'\ar[r] & B' \ar[r] & C' .}$$
Then, $ker(h)\simeq coker(f+\lambda)$, where $f+\lambda: A\oplus D\longrightarrow A'$ is the induced map.

\end{lemma}

\begin{theorem}\label{teoballesteros}
    Let $f:X\longrightarrow Y$ be a proper map from a connected generalized $n$-manifold to a connected generalized $(n+1)$-manifold with $H_1(Y;\mathbb{Z}_2)=0$ and let $A$ be the closure of the selfintersection set $A(f)= \{ x\in X : f^{-1}f(x)\neq x\}$. Suppose that $A\neq X$ and $Y-f(A)$ are connected. Then $\beta_0 (Y-f(X))= 2+ dim_{\mathbb{Z}_2}coker(i^{\ast}+f|_A^{\ast})$, where $i^{\ast}+f|_A^{\ast}: \check{H}_c^{n-1}(X;\mathbb{Z}_2)\oplus \check{H}_c^{n-1}(f(A);\mathbb{Z}_2) \longrightarrow \check{H}_c^{n-1}(A;\mathbb{Z}_2)$ is the induced map.
\begin{proof}
    To simplify the notation, we shall omit the coefficient group $\mathbb{Z}_2$ in all homology and cohomology groups.\newline
    \par Since $f$ is proper (and hence closed), $f(A)$, $f(X)$ are closed and we can consider the \v Cech cohomology of the pairs $(X,A)$, $(f(X),f(A))$ and get the following commutative diagram, where the rows are exact:
    $$\xymatrix{\ar[r] &\check{H}_c^{n-1}(f(X))\ar[r]\ar[d]^{(1)}& \check{H}_c^{n-1}(f(A))\ar[r]\ar[d]^{(2)}& \check{H}_c^{n}(f(X),f(A))\ar[r]\ar[d]^{(3)} &\\ \ar[r] &\check{H}_c^{n-1}(X)\ar[r]& \check{H}_c^{n-1}(A)\ar[r]& \check{H}_c^{n}(X,A)\ar[r]&}$$
    $$\xymatrix{\ar[r]&\check{H}_c^{n}(f(X))\ar[r]\ar[d]^{(4)}& \check{H}_c^{n}(f(A))\ar[r]\ar[d]^{(5)}& \check{H}_c^{n+1}(f(X),f(A))\ar[r]\ar[d]^{(6)} &\\ \ar[r]&\check{H}_c^{n}(X)\ar[r]& \check{H}_c^{n}(A)\ar[r]& \check{H}_c^{n+1}(X,A)\ar[r]&}$$
    But some of these cohomology groups are computed using the Alexander's duality:\\
    $\check{H}_c^n (X)\simeq H_0(X)\simeq \mathbb{Z}_2$, \\
    $\check{H}_c^n (f(X))\simeq H_1(Y,Y-f(X))\simeq \widetilde{H}_0(Y-f(X))$, 
    where the last isomorphism comes from the exact sequence of the pair $(Y,Y-f(X))$:\\
    $$\xymatrix{0=H_1(Y)\ar[r]& H_1(Y,Y-f(X))\ar[r]& \widetilde{H}_0(Y-f(X))\ar[r]& \widetilde{H}_0(Y)=0}.$$\\
    This gives a formula for the number of connected components of $Y-f(X)$: 
    \begin{equation}\label{eq}
    \beta_0(Y-f(X))= 1+ dim_{\mathbb{Z}_2}\check{H}_c^n(f(X)).
    \end{equation}
    \par We apply also the Alexander's duality to $A$ and $f(A)$:\\
    $\check{H}_c^{n}(A)\simeq H_0(X,X-A)=0,$\\
    $\check{H}_c^{n}(f(A))\simeq H_1(Y,Y-f(A))=0,$
    where the last equality comes from the exact sequence of the pair $(Y,Y-f(A))$:
    $$\xymatrix{0=H_1(Y)\ar[r]& H_1(Y,Y-f(A))\ar[r]& \widetilde{H}_0(Y-f(A))=0}.$$
    On the other hand, for Theorem 5 in \cite{Spanier} pag. 318, the maps $(3)$ and $(6)$ in the above diagram are isomorphisms. Then we can apply the Five Lemma to the maps $(2), \dots, (6)$ and deduce that $f^{\ast}:\check{H}_c^{n}(f(X))\longrightarrow \check{H}_c^{n}(X)$ is an epimorphism. Therefore $dim_{\mathbb{Z}_2}\check{H}_c^{n}(f(X))=1+dim_{\mathbb{Z}_2} ker(f^{\ast})$.\\
    \par But the above lemma \ref{lemaball} implies that $ker(f^{\ast})=coker(i^{\ast}+f|_A^{\ast})$, then $\beta_0 (Y-f(X))= 2+ dim_{\mathbb{Z}_2}coker(i^{\ast}+f|_A^{\ast})$.

\end{proof}    
\end{theorem}

We now consider some particular results depending on the closure of the self-intersection set $A$. If $A=\emptyset$, we obtain the following version of the Jordan–Brouwer Theorem.

\begin{corollary}[Jordan-Brouwer Theorem to generalized manifold]\label{jordan-brouwer}
    Let $f:X\longrightarrow Y$ be a proper embedding from a connected generalized $n$-manifold to a connected generalized $(n+1)$-manifold with $H_1(Y;\mathbb{Z}_2)=0$. Then the number of connected components of $Y-f(X)$ is 2.
    \begin{proof}
        The self-intersection set $A$ is empty, then $X\neq A$, $Y-f(A)$ is connected, $\check{H}_c^{n-1}(A;\mathbb{Z}_2)=0$ and $dim_{\mathbb{Z}_2}coker(i^{\ast}+f|_A^{\ast})=0$.
    \end{proof}
\end{corollary}

An other result is in the case where dim $A<n$ and $A$ is compact.

\begin{proposition}
    Let $f:X\longrightarrow Y$ be a proper map from a connected generalized $n$-manifold to a connected generalized $(n+1)$-manifold with $H_1(Y;\mathbb{Z}_2)=0$ and let $A$ be the closure of the selfintersection set $A(f)= \{ x\in X : f^{-1}f(x)\neq x\}$. Suppose that $dim\ A<n$ and $A$ is compact, then $Y-f(X)$ is disconnected.
    \begin{proof}
        As dim $A<n$ and $A$ is compact we have that $\check{H}_c^n(A)= \check{H}^n(A)=H^n(A)=0$, and hence $(5)$ in the diagram  $$\xymatrix{\ar[r] &\check{H}_c^{n-1}(f(X))\ar[r]\ar[d]^{(1)}& \check{H}_c^{n-1}(f(A))\ar[r]\ar[d]^{(2)}& \check{H}_c^{n}(f(X),f(A))\ar[r]\ar[d]^{(3)} &\\ \ar[r] &\check{H}_c^{n-1}(X)\ar[r]& \check{H}_c^{n-1}(A)\ar[r]& \check{H}_c^{n}(X,A)\ar[r]&}$$
    $$\xymatrix{\ar[r]&\check{H}_c^{n}(f(X))\ar[r]\ar[d]^{(4)}& \check{H}_c^{n}(f(A))\ar[r]\ar[d]^{(5)}& \check{H}_c^{n+1}(f(X),f(A))\ar[r]\ar[d]^{(6)} &\\ \ar[r]&\check{H}_c^{n}(X)\ar[r]& \check{H}_c^{n}(A)\ar[r]& \check{H}_c^{n+1}(X,A)\ar[r]&}$$
        is an epimorphism. Then the Five Lemma again to the maps $(2),\dots,(6)$ implies that $(4)$ is an epimorphism. Then we have that $\check{H}_c^n(f(X))\neq0$ and by equation \ref{eq}, $\beta_0(Y-f(X))\geq 2$.
    \end{proof}
\end{proposition}

\section{Separation by codimension-1 map depending on the primary obstruction to topological embedding}

In this section, we present the second separation result for codimension-1 maps into generalized manifolds, Theorem \ref{teoobstruccion}, which relies on the concept of the primary obstruction to topological embeddings as defined and studied in \cite{Biasi}\cite{Biasi1}\cite{Biasi2}\cite{Biasi3}\cite{Biasi4}. This result is analogous to the separation result by codimension-1 maps obtained by Biasi and Saeki in \cite{Biasi3} the context of smooth manifolds.

\par In the following of this section $R=\mathbb{Z}_2$, let $M$ and $N$ be generalized manifolds of dimensions $m$ and $n$, respectively, such that $k=n-m>0$ and $f:M\rightarrow N$ is a proper map. Let $A$ be the closure of the self-intersection set $A(f)= \{ x\in M : f^{-1}f(x)\neq x\}$. Let denote by $U_f\in H^k(N)$ the Poincar\'e dual of $f_\ast[M]\in H_m^c(N)$; in other words $f_\ast[M]=U_f\frown [N]$. Note that $f_\ast[M]$ is well-defined, since $f$ is a proper map.
\par Let the total Stiefel-Whitney class of $M$ and $N$ denoted by $w(M)\in H^\ast(M)$ and $w(N)\in H^\ast(N)$ respectively and let $\bar{w}(M)\in H^\ast(M)$; i.e. $\bar{w}(M)=w(M)^{-1}$.

\begin{definition}
    $w(f)=(f^\ast(w(N))\smile \bar{w}(M)$ is called total Stiefel-Whitney class of the stable normal bundle of $f$. We denote by $w_k(f)\in H^k(M)$ the degree $k$ term of $w(f)$, with is the $k$-th Stiefel-Witney class of stable normal bundle of $f$.
\end{definition}

\begin{definition}
   $ \theta(f)=(f^\ast U_f-w_k(f))\frown [M]\in H_{m-k}^c(M)$ is called primary obstruction to topological embedding.
\end{definition}
\begin{remark}
    This homology class is a proper homotopy invariant of
$f$ and has the property that, when $M$ and $N$ are generalized manifolds, if $f$ is properly homotopic to a embedding, then $\theta(f)$ vanishes, this was showed in \cite{Biasi}.
\end{remark}

Note that when $M$ is compact, $\theta(f)\in H_{m-k}(M)$, since we have $H_{m-k}(M)\simeq H_{m-k}^c(M)$.

Next three results are Theorem 3.1, Corollary 3.15 and Corollary 3.17 in \cite{Biasi}.

\begin{theorem}
    
 Let $f: M \rightarrow N$ be a proper map of an generalized $m$-manifold $M$ into an generalized $m+k$-manifold $N$ with $k>0$. Then $f_{*} \theta(f) \in H_{m-k}^{c}(N)$ always vanishes.
\end{theorem}
\begin{corollary}\label{coro3.15}
 Let $f: M \rightarrow N$ be a map from an compact generalized $m$-manifold $M$ into an generalized $m+k$-manifold $N$ with $k>0$. Set $B=f(A)$.
Then there exists an element $\mu \in \check{H}_{m-k}(A)$ such that $j_{\ast} \mu=\theta(f) \in \check{H}_{m-k}(M)=H_{m-k}(M)$ and $\left(f|_{A}\right)_{\ast}(\mu)=0 \in \check{H}_{m-k}(B)$, where $j: A \rightarrow M$ is the inclusion map. (When $A=\phi$, we regard $\check{H}_{m-k}(A)=0=\check{H}_{m-k}(B)$ )
\end{corollary}

\begin{corollary}\label{corollary3.17}
    Let $f:M\rightarrow N$ be a map from an compact generalized $m$-manifold $M$ into an generalized $(m+k)$-manifold with $k\geq 0$. If the topological dimension of $A$ is strictly less than $m-k$, then $\theta(f)\in H_{m-k}(M)$ vanishes.
\end{corollary}
\begin{lemma}\label{lema6.3}
 Let $f:M\rightarrow N$ be a map from an compact generalized $m$-manifold $M$ into an generalized $(m+k)$-manifold with $k\geq 0$. Suppose that $\mu \in \check{H}_{m-k}(A)$ is not zero. Then if $\theta (f)$ vanishes, $\bar{f}_\ast : \check{H}_{m-k+1}(M) \rightarrow \check{H}_{m-k+1}(f(M))$ is not subjective.
\begin{proof}
    Suppose that $f$ is not an embedding; i.e., $A\neq \emptyset$. Consider the following diagram of \v Cech homologies with exact rows:
    $$\xymatrix{\cdots\ar[r]& \check{H}_i(A)\ar[r]\ar[d] & \check{H}_i(M)\ar[r]\ar[d] & \check{H}_i(M,A)\ar[r]\ar[d]& \check{H}_{i-1}(A)\ar[r]\ar[d] &\cdots \\ \cdots\ar[r] &\check{H}_i(B)\ar[r] &\check{H}_i(f(M))\ar[r] &\check{H}_i(f(M),B)\ar[r] &\check{H}_{i-1}(B)\ar[r]& \cdots}$$
    where the vertical homomorphisms are induced by $f$. Note that the homomorphism $f_\ast :\check{H}_i(M,A)\rightarrow \check{H}_i(f(M),B) $ is an isomorphism by excision. Then it is not difficult to extract the following exact sequence:
    
    $$\check{H}_{m-k+1}(A) \stackrel{\alpha}{\rightarrow} \check{H}_{m-k+1}(B) \oplus \check{H}_{m-k+1}(M) \stackrel{((j'')_\ast\oplus\bar{f}_{\ast})}{\rightarrow} \check{H}_{m-k+1}(f(M))$$
    $${\rightarrow}\check{H}_{m-k}(A) \stackrel{\alpha}{\rightarrow} \check{H}_{m-k}(B) \oplus \check{H}_{m-k}(M)$$
where $j'': B \rightarrow f(M)$ is the inclusion map, $\alpha=((f|_{A})_\ast,i_\ast)$ where $i:A\rightarrow M$ is the inclusion map. By corollary \ref{coro3.15} and $\theta(f)=0$, then $\ker(\alpha)\neq 0$. Therefore $\bar{f}_\ast :\check{H}_{m-k+1}(M)\rightarrow \check{H}_{m-k+1}(f(M))$ can not be subjective.
\end{proof}
\end{lemma}
\begin{theorem}\label{teoobstruccion}
    
 Let $f: M \rightarrow N$ be a map from a compact generalized $m$-manifold $M$ into a generalized $(m+1)$-manifold $N$ with $H_{1}(N)=0$. If $A \neq M,\ \mu \neq 0$ and $w_{1}(f)=0$, then the number of connected components of $N-f(M)$ is greater than or equal to three.

\begin{proof}
    
 By hypothesis $A \neq M$, then $0=H_{0}(M, M-A)$, by Alexander's duality $\check{H}_{c}^{m}(A) \simeq 0$, how $M$ is compact hence $A$ compact, $H^m(A)=\check{H}^{m}(A)=\check{H}_{c}^{m}(A)=0$ and by universal coefficient theorem $0=\check{H}_{m}(A)=H_{m}(A)$.

Consider the following exact sequence of the pair $(N, N-f(M))$
$$
H_{1}(N) \rightarrow H_{1}(N, N-f(M)) \rightarrow \tilde{H}_{0}(N-f(M)) \rightarrow \tilde{H}_{0}(N)
$$
by hypothesis $H_{1}(N)=0=\tilde{H}_{0}(N)$, then $H_{1}(N, N-f(M))$ is isomorphic to $\widetilde{H}_{0}(N-f(M))$.
On the other hand $H_{1}(N, N-f(M))$ is isomorphic to $\check{H}^{m}(f(M))$ by Alexander's Duality. Thus we see that $\check{H}^{m}(f(M)) \simeq H^{m}(f(M))$ (because $f(M)$ is compact) which is isomorphic to $H_{m}(f(M))$ by universal coefficient theorem, then $\beta_{0}(N-f(M))=\operatorname{dim} H_{m}(f(M))+1$.

On the other hand we have the following exact sequence 
$$\check{H}_{m}(A) \stackrel{\alpha}{\rightarrow} \check{H}_{m}(B) \oplus \check{H}_{m}(M) \stackrel{((j'')_\ast\oplus\bar{f}_{\ast})}{\rightarrow} \check{H}_{m}(f(M))$$

where $j'': B \rightarrow f(M)$ is the inclusion map, $\alpha=((f|_{A})_\ast,i_\ast)$ where $i:A\rightarrow M$ is the inclusion map.

As $\check{H}_{m}(A)=0$, then $\tilde{f}_{\ast}: \check{H}_{m}(M) \rightarrow \check{H}_{m}(f(M))$ is a monomorphism. By our hypothesis $w_{1}(f)=0$ and $f^{\ast} U_{f}=0$ because $H_{1}(N)=0$ then $\theta(f)=0$ and by hypothesis $\mu\neq 0$. Therefore for lemma \ref{lema6.3} $f_{\ast}: \check{H}_{m}(M) \rightarrow \check{H}_{m}(f(M))$ is not subjective , then $\operatorname{dim} \check{H}_m(f(M)) \geqslant \operatorname{dim} \check{H}_{m}(M)+1=2 $, consequently $ \beta_{0}(N-f(M)) \geqslant 3$
\end{proof}
\end{theorem}

\end{document}